\documentclass[11pt,a4paper]{article}
\usepackage{mathrsfs,amsthm,amsmath,amsfonts,amssymb,mathrsfs,ascmac,
bm,enumerate,hyperref,graphicx,fullpage}
\usepackage[usenames]{color}
\usepackage{graphicx}
\numberwithin{equation}{section}
\theoremstyle{plain}
\newtheorem{thm}{Theorem}

\newtheorem{prop}[thm]{Proposition}

\theoremstyle{definition}
\newtheorem{defi}[thm]{Definition}

\newtheorem{rem}[thm]{Remark}
 
\newcommand{\C}{\mathbb{C}}
\newcommand{\R}{\mathbb{R}}

\newcommand{\D}{\mathbf{D}}

\newcommand{\mst}{\mathrm{m}}
\newcommand{\bst}{\mathrm{b}}
\newcommand{\im}{{\rm Im}}

\begin{document}
\title{Unimodality of Boolean and monotone stable distributions}
\author{Takahiro Hasebe\thanks{Supported by Marie Curie International Incoming Fellowships PIIF-GA-2012-328112.} \\ Laboratoire de Math\'{e}matiques, 
Universit\'{e} de Franche-Comt\'{e}\\
16 route de Gray
25030 Besan\c{c}on cedex, 
France\\
thasebe@univ-fcomte.fr
\and Noriyoshi Sakuma\thanks{Supported by JSPS KAKENHI Grant Number 24740057.} \\Department of Mathematics, Aichi University of Education\\
1 Hirosawa, Igaya-cho, Kariya-shi, 448-8542, Japan\\
sakuma@auecc.aichi-edu.ac.jp}
\date{\today}
\maketitle
\abstract{We give a complete list of the Lebesgue-Jordan decomposition of Boolean and monotone stable distributions and a complete list of the mode of them. They are not always unimodal.}
\section{Introduction} 
Boolean and monotone stable distributions were defined in \cite{SW97,H10} in the context of non-commutative probability theory with Boolean and monotone independence, respectively \cite{SW97,M01}. 
The aspect of domains of attraction for these distributions are studied in \cite{W12,AW13} and \cite{BP99}, respectively. 

They also play important roles in free probability theory: Positive monotone stable laws are the marginal laws of a free L\'evy process of second kind \cite[Theorem 4.5, Corollary 4.5]{B98};  A compound free Poisson distribution having a monotone stable law as its free L\'evy measure has  explicit Cauchy and Voiculescu transforms \cite{AH13}; A positive Boolean stable distribution is the law of the quotient of two i.i.d.\ classical stable random variables, and at the same time, it is the law of the quotient of i.i.d.\ free stable random variables (in the free sense) \cite{BP99,AH}.

In this paper, we first determine the absolutely continuous part and also the singular part of the monotone and Boolean stable laws. While part of this computation is known in the literature \cite{AH13,AH,H10}, there is no complete list of the formulas.  

Second, we investigate the mode of these measures. It is known that classical and free stable distributions are unimodal~\cite{Y78, BP99}. 
More generally, selfdecomposable and free selfdecomposable distributions, which respectively include all stable and free stable distributions, are unimodal \cite{Y78,HT}. 
However, monotone stable laws and Boolean stable laws include the arcsine law and the Bernoulli law respectively, and we cannot expect unimodality for all. We obtain the mode of all monotone and Boolean stable distributions. 
 
\section{Unimodality of Boolean and monotone stable distributions}
First, we gather analytic tools and their properties to compute Boolean and monotone stable distributions. 
\subsection{Analytic tools}
Let $\mathcal{P}$ denote the set of all Borel probability measures 
on $\mathbb{R}$. 
In the following, we explain main tools of free probability. 
Let $\C^+ : = \{ z \in \C : \mathrm{Im} (z) > 0 \}$ and $\C^- : = \{ z \in \C : \mathrm{Im} (z) < 0 \}$.
For $\mu \in \mathcal{P}$, 
the Cauchy transform $G_{\mu}:\mathbb{C}^{+}\rightarrow\mathbb{C}
^{-}$ is defined by
\[
G_{\mu}(z)=\int_{\mathbb{R}}\frac{1}{z-x}\mu(\mathrm{d}x),\quad z\in
\mathbb{C}^{+}.
\]
and the reciprocal Cauchy transform $F_{\mu}:\C^{+}\to\C^{+}$of $\mu \in \mathcal{P}$ is defined by
\[
F_{\mu}(z)=\frac{1}{G_{\mu}(z)},\quad z \in \mathbb{C}^+.
\]
In this paper, we apply the Stieltjes inversion formula \cite{A65,T00} for Boolean and monotone stable distributions. For any Borel probability measure $\mu$, we can recover the distribution from its Cauchy transform: if $\mu$ does not have atoms at $a,b$, we have
\begin{align*}
\mu([a,b]) = -\frac{1}{\pi} \lim_{y\searrow 0}\im \int_{[a,b]} G_{\mu}(x+iy) dx.
\end{align*}
Especially, if $G_\mu(z)$ extends to a continuous function on $\C^+ \cup I$ for an open interval $I \subset \R$, then the distribution $\mu$ has continuous derivative $f_{\mu}=d\mu/dx$ with respect to the Lebesgue measure $dx$ on $I$, and  
we obtain $f_{\mu}(x)$ by
\begin{align*}
f_{\mu}(x)=-\frac{1}{\pi} \lim_{y\searrow 0}\im\, G_{\mu}(x+iy),\quad x\in I.
\end{align*} 
Atoms of $\mu$ may be computed by the formula
\begin{align*}
\mu(\{ a\}) = \lim_{z\to a, \,z\in\C^{+}}(z-a)G_{\mu}(z) = \lim_{y\searrow 0}y\, G_{\mu}(a+iy), \quad a\in\R. 
\end{align*}

\subsection{Boolean case} 
In this paper, the maps $z\mapsto z^p$ and $z\mapsto\log z$ always denote the principal values for $z\in\C\setminus (-\infty,0]$.  
Correspondingly $\arg (z)$ is defined in $\C\setminus (-\infty,0]$ so that it takes values in $(-\pi,\pi)$. 

\begin{defi} Let $\bst_{\alpha,\rho}$ be a boolean stable law~\cite{SW97} characterized by the following. 
\begin{enumerate}[{\rm(1)}]
\item If $\alpha \in(0,1)$, then 
$$
F_{\bst_{\alpha,\rho}}(z)=z+e^{i\rho \alpha\pi}z^{1-\alpha},  \qquad z\in \C^+,\qquad \rho \in [0,1].  
$$
\item If $\alpha =1$, then 
$$
F_{\bst_{\alpha,\rho}}(z)=z+2\rho i -\frac{2(2\rho-1)}{\pi}\log z, \qquad z\in \C^+,\qquad \rho \in[0,1].  
$$
\item If $\alpha\in(1,2]$, then 
$$
F_{\bst_{\alpha,\rho}}(z)=z+ e^{i[(\alpha-2)\rho+1]\pi} z^{1-\alpha},\qquad z\in\C^+,\qquad \rho \in [0,1]. 
$$
\end{enumerate}
\end{defi}
\begin{rem}
The case $\alpha=1$ includes non strictly stable distributions which were considered in \cite{AH}.  
\end{rem}

In the case $\alpha\in(0,1)\cup(1,2]$, for simplicity we also use a parameter $\theta$, instead of $\rho$, defined by  
\begin{equation}\label{theta}
\theta=
\begin{cases}
\rho\alpha\pi\in[0,\alpha\pi],& \alpha\in(0,1),\\
[(\alpha-2)\rho+1]\pi\in[(\alpha-1)\pi,\pi],&\alpha\in(1,2]. 
\end{cases}
\end{equation}

The probability measure $\bst_{\alpha,\rho}$ is described as follows. Let 
\begin{align}
&B_\alpha(x, \theta)=
\frac{\sin \theta}{\pi}\frac{x^{\alpha-1}}{x^{2\alpha} +2 x^\alpha \cos \theta + 1},& x>0,&\qquad \alpha\in(0,1)\cup (1,2), \\
&B_1(x, \rho)=
\frac{2\rho}{\pi} \frac{1}{(x-\frac{2(2\rho-1)}{\pi}\log x)^2+4\rho^2},& x>0. 
\end{align}

\begin{prop}\label{boolean stable law}
The Boolean stable distributions are as follows. 
\begin{enumerate}[{\rm(1)}]
\item\label{b1} If $\alpha\in(0,1)$ and $\rho\in(0,1)$, then 
\[
\bst_{\alpha,\rho}(dx)=B_\alpha(x,\rho\alpha\pi)1_{(0,\infty)}(x)\,dx + B_\alpha(-x,(1-\rho)\alpha\pi)1_{(-\infty,0)}(x)\,dx. 
\]
\item\label{b2} If $\alpha \in(0,1)$ and $\rho=1$, then 
\[
\bst_{\alpha,1}(dx) = B_\alpha(x,\alpha\pi)1_{(0,\infty)}(x)\,dx. 
\]



\item\label{b4} If $\alpha=1$ and $\rho\in(0,1)$, then 
\[
\begin{split}
\bst_{1,\rho}(dx)
&=B_1(x,\rho)1_{(0,\infty)}(x)\,dx +B_1(-x,1-\rho)1_{(-\infty,0)}(x)\,dx. 
\end{split}
\]
The case $\rho=\frac{1}{2}$ is the Cauchy distribution 
$$
\bst_{1,1/2}(dx)=\frac{1}{\pi(x^2+1)}1_{\R}(x)\,dx.  
$$

\item\label{b5} If $\alpha=1$ and $\rho=1$, then 
$$
\bst_{1,1}(dx)=B_1(x,1)1_{(0,\infty)}(x)\,dx+  \frac{u_+(0)}{u_+(0)+2/\pi}\delta_{-u_+(0)}, 
$$
where $u_+(0)=0.4745\dots$ is the unique solution $u$ of the equation $\pi u +2\log u=0$, $u\in(0,\infty)$.


\item\label{b6} If $\alpha\in(1,2)$ and $\rho\in(0,1)$, then
\[
\bst_{\alpha,\rho}(dx)=B_\alpha(x,[\rho\alpha -2\rho+1]\pi)1_{(0,\infty)}(x)\,dx + B_\alpha(-x,[(1-\rho)\alpha+2\rho-1]\pi)1_{(-\infty,0)}(x)\,dx. 
\]

\item\label{b7} If $\alpha\in(1,2)$ and $\rho=1$, then $\bst_{\alpha,\rho}$ has an atom at $-1$:   
\[
\bst_{\alpha,1}(dx) = 
B_\alpha(x,(\alpha-1)\pi)1_{(0,\infty)}(x)\,dx +\frac{1}{\alpha}\delta_{-1}. 
\]

\item\label{b8} If $\alpha=2$ and $\rho\in[0,1]$, then 
\[
\bst_{2,\rho}=\frac{1}{2}(\delta_{-1}+\delta_1). 
\]

\end{enumerate}
For each $\alpha\in(0,2)$, the replacement $\rho\mapsto 1-\rho$ gives the reflection of the measure regarding $0$. 

\end{prop}
\begin{proof} 
The case $\alpha=2$ is well known and we omit it. 

(i)\,\, The case $\alpha\neq1$. 
We write $G(z),F(z)$ instead of $G_{\bst_{\alpha,\rho}}(z), F_{\bst_{\alpha,\rho}}(z)$ respectively. 
If $x>0$, then $\lim_{y\searrow0}(x+iy)^{1-\alpha}$ is simply $x^{1-\alpha}$, while if $x<0$, then the argument of $x+iy$ approaches to $\pi$, and so $\lim_{y\searrow0}(x+iy)^{1-\alpha} = (-x)^{1-\alpha}e^{i\pi(1-\alpha)}=(-x)^{1-\alpha}e^{i\pi(1-\alpha)}$. 
So we have for $x>0$ that 
\[
\begin{split}
\lim_{y\searrow0}G(x+iy) 
&= \frac{1}{x +e^{i\theta}x^{1-\alpha}} = \frac{x^{\alpha-1}}{x^\alpha +e^{i\theta}} = \frac{x^{\alpha-1}(x^\alpha+\cos\theta-i\sin \theta)}{(x^\alpha+\cos \theta)^2+\sin^2\theta} 
\end{split}
\]
and for $x<0$
\[
\begin{split}
\lim_{y\searrow0}G(x+iy) 
&= \frac{1}{x -e^{i(\theta-\alpha\pi)}(-x)^{1-\alpha}} = -\frac{(-x)^{\alpha-1}}{(-x)^\alpha +e^{i(\theta-\alpha\pi)}} \\
&= -\frac{(-x)^{\alpha-1}(x+\cos(\theta-\alpha\pi)-i\sin(\theta-\alpha\pi))}{((-x)^\alpha+\cos (\theta-\alpha\pi))^2+\sin^2(\theta-\alpha\pi)}. 
\end{split}
\]
Taking the imaginary part of these expressions, we obtain
\begin{equation}\label{boolean density}
-\frac{1}{\pi}\lim_{y\searrow0}\im\,G(x+iy) 
=
\begin{cases}
\displaystyle\frac{\sin \theta}{\pi}\frac{x^{\alpha-1}}{x^{2\alpha} +2 x^\alpha \cos \theta + 1},& x>0, \\[12pt]
\displaystyle\frac{\sin (\alpha\pi-\theta)}{\pi}\frac{(-x)^{\alpha-1}}{(-x)^{2\alpha} +2 (-x)^\alpha \cos(\alpha\pi-\theta) + 1},&x<0. 
\end{cases}
\end{equation}

(i-1)\,\, The case $\alpha\in(0,1)$.
If $\alpha\in(0,1)$ and $\rho\in[0,1]$, the function $F$ extends continuously to $\C^+\cup \R$ and does not have a zero except at $z=0$ in $\C^+\cup \R$, but since $\lim_{y\searrow0}yG(iy)=0$, there is no atom at $0$. 
The above argument shows (\ref{b1}) and (\ref{b2}).  

(i-2)\,\, The case $\alpha\in(1,2)$. 
First note that the measure does not have an atom at $0$ since $\lim_{y\searrow0}G(iy)=0$. If $\theta\in((\alpha-1)\pi,\pi)$, then the same computation (\ref{boolean density}) is valid and so we obtain (\ref{b6}).  
If $\theta =(\alpha-1)\pi$ (or equivalently  $\rho=1$), then (\ref{boolean density}) is valid for $x>0$. Now note that $z\mapsto F(z)=z+e^{i(\alpha-1)\pi}z^{1-\alpha}=z+(-z)^{1-\alpha}$ has a zero $z=-1$ in $(-\infty,0)$. Hence $G$ extends to a continuous function on $\C^+\cup \R \setminus\{-1\}$, and we have $\im\,G(x+i0)=0$ for $x<0$, $x\neq-1$. There is an atom at $-1$ with weight $1/\alpha$ since 
$$
\lim_{z\to-1, z\in\C^+}(z+1)G(z)=\lim_{z\searrow-1, z\in\C^+}\frac{1}{z}\cdot\frac{1-(-z)}{1-(-z)^{-\alpha}}=\frac{1}{\alpha}. 
$$
This implies (\ref{b7}). 

(ii)\,\, The case $\alpha=1$.  We can easily see that $\lim_{y\searrow0}y G(iy)=0$ and so there is no atom at 0. Assume first that $\rho\in(0,1)$. For $x>0$, we have that 
\begin{equation}\label{rho(0,1)}
\begin{split}
\lim_{y\searrow0}\im\,G(x+iy) 
&= \lim_{y\searrow0}\im\,\frac{1}{(x+iy)+2\rho i -\frac{2(2\rho-1)}{\pi}\log (x+iy)} \\
&= \im\,\frac{1}{x -\frac{2(2\rho-1)}{\pi}\log x +2\rho i} \\
&= \im\,\frac{x -\frac{2(2\rho-1)}{\pi}\log x -2\rho i}{(x -\frac{2(2\rho-1)}{\pi}\log x)^2 +4\rho^2} \\
&= \frac{-2\rho}{(x -\frac{2(2\rho-1)}{\pi}\log x)^2 +4\rho^2} 
\end{split}
\end{equation}
and so we get $-\frac{1}{\pi}\lim_{y\searrow0}\im\,G(x+iy)=B_1(x,\rho)$. For $x<0$, note that $\log(x+i0) = \log (-x) +i\pi$ and then similarly to (\ref{rho(0,1)}) we get 
$-\frac{1}{\pi}\lim_{y\searrow0}\im\,G(x+iy)=B_1(-x,1-\rho)$. Since $G$ extends to a continuous function on $\C^+ \cup \R \setminus\{0\}$, we get the formula (\ref{b4}) by the Stieltjes inversion.

If $\rho=1$, then the computation for $x>0$ is the same as (\ref{rho(0,1)}). For $x<0$, note that 
$$
F(x+i0)= x+2i -\frac{2}{\pi}(\log(-x)+i\pi) = -((-x) +\frac{2}{\pi}\log(-x)),  
$$
which has the unique zero at $x=-u_+(0)$. So $G$ extends to a continuous function on $\C^+ \cup \R \setminus \{0,-u_+(0)\}$ and $\im\, G(x+i0)=0$ for $x<0, x \neq -u_+(0)$. 
We have the series expansion 
$$
x -\frac{2}{\pi}\log(-x)= a(x+u_+(0)) + b(x+u_+(0))^2 + \cdots. 
$$
Then $a=\frac{d}{dx}\Big|_{x=-u_+(0)}(x -\frac{2}{\pi}\log(-x)))= \frac{u_+(0)+2/\pi}{u_+(0)}$. The weight of the atom at $-u_+(0)$ is equal to $1/a$, and so we have (\ref{b5}). 

(iv)\,\, The final statement is proved as follows. For $\alpha\in(0,2)$ and $\rho\in(0,1)$, the measure $\bst_{\alpha,\rho}$ is the reflection of $\bst_{\alpha,1-\rho}$ regarding $x=0$ by the formulas (\ref{b1}), (\ref{b4}) and (\ref{b6}). For fixed $\alpha\in(0,2)$, the measure $\bst_{\alpha,\rho}$ depends on $\rho$ weakly continuously. Hence, by approximation, the reflection property is true for any $\rho\in[0,1]$. 
\end{proof}

\begin{thm} 
Let $\alpha_0=0.7364\ldots$ be the unique solution of the equation $\sin(\pi \alpha)=\alpha$, $\alpha\in(0,1)$. Let 
\begin{align}
&x_+:= \left(\frac{-\cos\theta+ \sqrt{\alpha^2-\sin^2\theta}}{1+\alpha}\right)^{1/\alpha},  \\
&x_-:= -\left(\frac{-\cos(\alpha\pi-\theta)+ \sqrt{\alpha^2-\sin^2(\alpha\pi-\theta)}}{1+\alpha}\right)^{1/\alpha}. 
\end{align}
Let $u_+=u_+(\rho)$ be the unique solution $x$ of the equation $\pi x +2(1-2\rho)\log x=0$, $x\in(0,\infty)$, $\rho\in[0,\frac{1}{2})$, and let 
 $u_-=u_-(\rho):=-u_+(1-\rho)$, $\rho\in(\frac{1}{2},1]$. 
\begin{enumerate}[\rm(1)]
\item\label{bu1} If $\alpha \in (0,\alpha_0]$, then $\bst_{\alpha,\rho}$ is unimodal with mode 0. 

\item\label{bu2} If $\alpha \in (\alpha_0,1)$, then there are sub cases. 
\begin{enumerate}[\rm (a)] 
\item\label{bu2b} If $\theta\in[0,\arcsin(\alpha)-(1-\alpha)\pi)$, then $\bst_{\alpha,\rho}$ is bimodal with modes 0 and $x_+$. 
\item\label{bu2a} If $\theta\in[\arcsin(\alpha)-(1-\alpha)\pi, \pi -\arcsin(\alpha)]$, then $\bst_{\alpha,\rho}$ is unimodal with mode 0.  

\item\label{bu2c} If $\theta\in(\pi -\arcsin(\alpha),\alpha\pi]$, then $\bst_{\alpha,\rho}$ is bimodal with modes $x_-$ and $0$. 
\end{enumerate}

\item\label{bu3} If $\alpha =1$, then there are sub cases. 
\begin{enumerate}[\rm (a)] 

\item\label{bu3b}  If $\rho\in[0,\frac{1}{2})$, then $\bst_{\alpha,\rho}$ is bimodal with modes $-\frac{2(1-2\rho)}{\pi}$ and $u_+$. If $\rho=0$, then the mode at $u_+$ is an atom.  
\item\label{bu3a} If $\rho=\frac{1}{2}$, then $\bst_{\alpha,\rho}$ is unimodal with mode 0.  

\item\label{bu3c}  If $\rho\in(\frac{1}{2},1]$, then $\bst_{\alpha,\rho}$ is bimodal with modes $u_-$ and $\frac{2(2\rho-1)}{\pi}$. If $\rho=1$, then the mode at $u_-$ is an atom.  
\end{enumerate}

\item\label{bu4}  If $\alpha\in(1,2]$, then $\bst_{\alpha,\rho}$ is bimodal with modes $x_-$ and $x_+$. If $\theta=(\alpha-1)\pi, \pi$, then $\bst_{\alpha,\rho}$ has an atom at $x_-=-1,x_+=1$, respectively.

\end{enumerate}
\end{thm} 
\begin{proof} (i)\,\, The case $\alpha\in(0,\alpha_0]$. Note that $0$ is a mode since the density diverges to $\infty$ at $x=0$. We can easily compute 
\begin{equation}\label{derivative boolean density}
\frac{\partial}{\partial x} B_\alpha(x,\theta) = -\frac{x^{\alpha-2}\sin\theta}{\pi}\cdot\frac{ (1+\alpha)(x^\alpha+\frac{1}{1+\alpha} \cos \theta)^2 +\frac{1}{1+\alpha}(\sin^2\theta -\alpha^2)}{\left(x^{2 \alpha}+2 x^\alpha \cos \theta+1\right)^2},\quad x>0. 
\end{equation}
Let 
$$
f(x):=(1+\alpha)\left(x^\alpha+\frac{1}{1+\alpha} \cos \theta\right)^2 +\frac{1}{1+\alpha}(\sin^2\theta -\alpha^2). 
$$
Note that $f(0)=1-\alpha>0$. Since $f(x)$ is a polynomial on $x^\alpha$ of degree 2, it is easy to see that if $\theta \in[0,\frac{\pi}{2}]$, then $\cos\theta\geq0$ and $f(x)$ does not have a zero in $(0,\infty)$. If $\theta \in (\frac{\pi}{2}, \alpha\pi)$, then $f(x)$ attains a local minimal value at $x=-\frac{1}{1+\alpha}\cos\theta$, but now $\sin^2\theta-\alpha^2 \geq \sin^2(\alpha\pi)-\alpha^2 \geq0$ for $\alpha \leq \alpha_0$. Hence $f(x)\geq0$ for $x>0$ and the map $x\mapsto B_\alpha(x,\theta)$ is strictly decreasing on $(0,\infty)$ for any $\theta\in(0,\alpha\pi]$. By the reflection property (see the last statement of Proposition \ref{boolean stable law}) the density is strictly increasing on $(-\infty,0)$ for any $\theta \in[0,\alpha\pi)$ and hence $\bst_{\alpha,\rho}$ is unimodal, the conclusion (\ref{bu1}). 

(ii)\,\, The case $\alpha\in(\alpha_0,1)$. In this case $0$ is still a mode of $\bst_{\alpha,\rho}$. 
From (\ref{derivative boolean density}), we have that 
\begin{equation}\label{equivalence1}
\begin{split}
\text{$x\mapsto B_\alpha(x,\theta)$ takes a local maximum in $(0,\infty)$} 
&\Leftrightarrow \text{$\theta\in\left(\frac{\pi}{2},\alpha\pi\right]$ and $\sin\theta < \alpha$}\\
&\Leftrightarrow \text{$\theta\in(\pi-\arcsin(\alpha),\alpha\pi]$},  
\end{split}
\end{equation}
and if this condition is satisfied, then the local maximum is attained at $x=x_+$. By reflection, it holds that 
\begin{equation}\label{equivalence2}
\begin{split}
\text{$x\mapsto B_\alpha(-x,\alpha\pi-\theta)$ takes a local maximum in $(-\infty,0)$} &\Leftrightarrow \text{$\theta\in[0,\arcsin(\alpha)-(1-\alpha)\pi)$}, 
\end{split}
\end{equation}
and if this condition is satisfied, the local maximum is attained at $x=x_-$. The two conditions (\ref{equivalence1}) and (\ref{equivalence2}) cannot be satisfied for the same $\theta$. Hence we have the conclusion (\ref{bu2}). 

(iii)\,\, The case $\alpha\in(1,2)$. Note that $0$ is not a mode of $\bst_{\alpha,\rho}$ since the density function takes $0$ at $x=0$. Since now $f(0)=1-\alpha<0$, we conclude that the map $x\mapsto B_\alpha(x,\theta)$ in $(0,\infty)$ takes a unique local maximum  at $x=x_+$ for any $\theta \in[(\alpha-1)\pi,\pi)$. By reflection and by Proposition \ref{boolean stable law}(\ref{b6})--(\ref{b8}), the conclusion follows. 

(iv)\,\, The case $\alpha=1$, $\rho\neq\frac{1}{2}$. Note that the density takes 0 at $x=0$. We have 
\begin{equation}\label{derivative boolean density 2}
\frac{\partial}{\partial x} B_1(x,\rho) =-\frac{4\rho\left(  x -\frac{2(2\rho-1)}{\pi}\right) \left(  x-\frac{2( 2 \rho-1)}{\pi} \log x\right)}{\pi x\left((x-\frac{2(2\rho-1)}{\pi}\log x)^2+4\rho^2\right)^2},\qquad x>0, 
\end{equation}
and then the remaining calculus is not difficult.  
\end{proof}
\begin{rem}
If $\alpha\in(0,1)$, the density $\bst_{\alpha,\rho}$ diverge at $0$. So, $x=0$ becomes always mode.
In addition, at $x=0$ the density cannot be differentiable. On the other hand, if $\alpha\in(\alpha_{0},1)$ and $\theta\in[0,\arcsin (\alpha)-(1-\alpha)\pi]$, at $x_{+}$ the density is differentiable.
\end{rem}

\subsection{Monotone case}
As before, the map $z\mapsto z^p$ denotes the principal value (defined in $\C\setminus (-\infty,0]$). We also use a different branch  
$$
z\mapsto z_{(0,2\pi)}^{p} = e^{ p \log |z| +ip \arg_{(0,2\pi)}z},   \qquad z\in\C\setminus [0,\infty), 
$$
where $ \arg_{(0,2\pi)}z$ is defined continuously so that $ \arg_{(0,2\pi)}z \in(0,2\pi)$. 

Let $\mst_{\alpha,\rho}$ be a monotone strictly stable law \cite{H10} characterized by 
\begin{defi}\label{definition monotone}
\begin{enumerate}[\rm(1)]
\item\label{} If $\alpha \in(0,1)$, then 
$$
F_{\mst_{\alpha,\rho}}(z)=(z^\alpha+e^{i\rho \alpha\pi})^{1/\alpha},  \qquad z\in \C^+,\qquad \rho \in [0,1].  
$$
\item If $\alpha =1$, then we only consider $\rho=\frac{1}{2}$: 
$$
F_{\mst_{1,1/2}}(z)=z+i, \qquad z\in \C^+.  
$$
\item If $\alpha\in(1,2]$, then 
$$
F_{\mst_{\alpha,\rho}}(z)=(z^\alpha+ e^{i[(\alpha-2)\rho+1]\pi})_{(0,2\pi)}^{1/\alpha}, \qquad z\in\C^+,\qquad \rho \in [0,1].
$$
\end{enumerate}
\end{defi}

\begin{rem} \begin{enumerate}[(a)]
\item If $\alpha\in(0,1)$, then $z^\alpha+e^{i\rho \alpha\pi}$ stays in $\C^+$ for $z\in\C^+$. However for $\alpha\in(1,2]$, $z^\alpha+ e^{i[(\alpha-2)\rho+1]\pi}$ may be in $\C^-$, and so we need to use the branch $(\cdot)_{(0,2\pi)}^{1/\alpha}$ to define $F_{\mst_{\alpha,\rho}}$ analytically (or continuously) in $\C^+$. 

\item The above definition does not respect the Bercovici-Pata bijection. If we hope to let $\mst_{\alpha,\rho}$ correspond to $\bst_{\alpha,\rho}$ regarding the monotone-Boolean Bercovici-Pata bijection, then we have to consider $\D_{\alpha^{1/\alpha}}(\mst_{\alpha,\rho})$ which is the induced measure of $\mst_{\alpha,\rho}$ by the map $x \mapsto \alpha^{1/\alpha}x$.  
\item All the above distributions are strictly stable. 
Non strictly stable distributions are not defined in the literature, and so we do not consider the non-symmetric case in $\alpha=1$. 
\end{enumerate}
\end{rem}
 
We will describe the probability measure $\mst_{\alpha,\rho}$. Let $\theta=\theta(\alpha,\rho)$ be (\ref{theta}) as before. For $\alpha\in(0,2], \theta \in(0,\pi)$,  let  
\[
M_{\alpha}(x,\theta)=\frac{\sin[\frac{1}{\alpha}\varphi(x^\alpha,\theta)]}{\pi(x^{2\alpha}+2x^\alpha \cos \theta +1)^{1/(2\alpha)}},\qquad x>0,
\]
where 
$$
\varphi(x, \theta)
=
\begin{cases}
\arctan(\frac{\sin\theta}{x+\cos\theta}),& x>-\cos\theta,\\
\frac{\pi}{2}, &x=-\cos \theta, \\
\arctan(\frac{\sin\theta}{x+\cos\theta})+\pi,& 0<x<-\cos\theta. 
\end{cases} 
$$
The second and the third cases do not appear if $\cos\theta \geq0$. We can also write $\varphi(x, \theta)=\arg(x + e^{i\theta})$. 

\begin{prop}\label{MD} The strictly monotone stable distributions are as follows. 
\begin{enumerate}[\rm(1)]
\item\label{m1} If $\alpha\in(0,1)$ and $\rho\in(0,1)$, then 
\[
\mst_{\alpha,\rho}(dx)=M_\alpha(x,\rho\alpha\pi)1_{(0,\infty)}(x)\,dx + M_\alpha(-x,(1-\rho)\alpha\pi)1_{(-\infty,0)}(x)\,dx. 
\]

\item \label{m2} If $\alpha \in(0,1)$ and $\rho=1$, then 
\[
\mst_{\alpha,1}(dx) = M_\alpha(x,\alpha\pi)1_{(0,\infty)}(x)\,dx. 
\]

\item\label{m3} If $\alpha=1$ and $\rho=\frac{1}{2}$, then 
$$
\mst_{1,1/2}(dx)=\frac{1}{\pi(x^2+1)}1_{\R}(x)\,dx.  
$$

\item\label{m4}  If $\alpha\in(1,2)$ and $\rho\in(0,1)$, then
\[
\mst_{\alpha,\rho}(dx)=M_\alpha(x,[\rho\alpha -2\rho+1]\pi)1_{(0,\infty)}(x)\,dx + M_\alpha(-x,[(1-\rho)\alpha+2\rho-1]\pi)1_{(-\infty,0)}(x)\,dx. 
\]

\item\label{m5} If $\alpha\in(1,2)$ and $\rho=1$, then    
\[
\mst_{\alpha,0}(dx) = 
M_\alpha(x,(\alpha-1)\pi)1_{(0,\infty)}(x)\,dx + \frac{\sin(\pi/\alpha)}{\pi(1-(-x)^\alpha)^{1/\alpha}}1_{(-1,0]}(x)\,dx.  
\]


\item\label{m6} If $\alpha=2$ and $\rho\in[0,1]$, then 
\[
\mst_{2,\rho}(dx)=\frac{1}{\pi\sqrt{1-x^2}}1_{(-1,1)}(x)\,dx. 
\]

\end{enumerate}
They are all absolutely continuous with respect to the Lebesgue measure. In the cases $(\alpha,\rho)\in(1,2)\times \{0,1\}$ and $\alpha=2$, the density function diverges to infinity at the edge of the support, but in the other cases the density function is either continuous on $\R$, or extends to a continuous function on $\R$ (if the support is not $\R$).  
The density function is real analytic except at the edge of the support and at 0. The replacement $\rho\mapsto 1-\rho$ gives the reflection of the measure around $x=0$. 

\end{prop}
\begin{proof} Let $G(z),F(z)$ denote $G_{\mst_{\alpha,\rho}}(z), F_{\mst_{\alpha,\rho}}(z)$ respectively. First note that $\mst_{\alpha,\rho}$ does not have an atom at $x=0$ since $\lim_{y\searrow0}F(iy)\neq0$. 

We will not consider the cases $\alpha=1$ and $\alpha=2$ since these are well known. 

(i)\,\, The case $\alpha\in(0,1)$, $\rho \in (0,1)$. For $x>0$, we have that 
\begin{equation}\label{eq010}
\lim_{y\searrow 0} (x+i y)^\alpha +e^{i\theta} = x^\alpha +\cos\theta +i\sin\theta \in \C^+,  
\end{equation}
which equals $r(x^\alpha,\theta)e^{i\varphi(x^\alpha,\theta)}$, where 
$$
r(x,\theta)= \sqrt{x^2 +2x \cos\theta+ 1}. 
$$ 
So we get 
\begin{equation}\label{monotone density a}
\begin{split}
-\frac{1}{\pi}\lim_{y\searrow0}\im\,G(x+iy) 
&= -\frac{1}{\pi}\im\,\frac{1}{(x^\alpha +e^{i\theta})^{1/\alpha}} = -\frac{1}{\pi}\im\, r(x^\alpha,\theta)^{-1/\alpha}e^{-i\varphi(x^\alpha,\theta)/\alpha}\\
&= \frac{1}{\pi}r(x^\alpha,\theta)^{-1/\alpha} \sin\left(\frac{1}{\alpha}\varphi(x^\alpha,\theta)\right), \quad x>0, 
\end{split}
\end{equation}
which is strictly positive since now $\theta \in(0,\pi)$. This implies that $\mst_{\alpha,\rho}$ is absolutely continuous on $(0,\infty)$ with respect to the Lebesgue measure, and the density function is given by $M_{\alpha}(x,\theta)$. 

Note that, for $x<0$, it holds that 
\begin{equation}\label{eq10}
\lim_{y\searrow 0} (x+i y)^\alpha +e^{i\theta} = (-x)^\alpha e^{i\alpha\pi} +e^{i\theta} = e^{i\alpha\pi}((-x)^\alpha + e^{i(\theta-\alpha\pi)}). 
\end{equation}
Now $\theta -\alpha\pi \in(-\pi,0).$ From geometric consideration, the following can be justified: 
\begin{equation}
\begin{split}
\left( e^{i\alpha\pi}((-x)^\alpha + e^{i(\theta-\alpha\pi)})\right)^{1/\alpha} 
&=  \left( e^{i\alpha\pi}\right)^{1/\alpha}\left((-x)^\alpha + e^{i(\theta-\alpha\pi)}\right)^{1/\alpha} \\
&=  -\left((-x)^\alpha + e^{i(\theta-\alpha\pi)})\right)^{1/\alpha},\qquad x<0. 
\end{split}
\end{equation}
We can show that 
\begin{equation}\label{formula1} 
(-x)^\alpha + e^{i(\theta-\alpha\pi)} =r((-x)^\alpha, \alpha\pi-\theta)e^{-i\varphi((-x)^\alpha,\alpha\pi-\theta)}, 
\end{equation}
and so we get (\ref{m1}) from a computation similar to (\ref{monotone density a}). 

(ii)\,\, The case $\alpha\in(0,1)$, $\rho=1$. The formula (\ref{monotone density a}) still holds for $x>0$. For $x<0$, we have 
$$
F(x+i0)=((x+i0)^\alpha+e^{i\alpha\pi})^{1/\alpha} = ((-x)^\alpha e^{i\alpha\pi}+e^{i\alpha\pi})^{1/\alpha} =-((-x)^\alpha +1)^{1/\alpha} <0,  
$$
and hence $\mst_{\alpha,1}$ does not have support on $(-\infty,0)$. So we get (\ref{m2}).  

(iii)\,\, The case $\alpha\in(1,2)$, $\rho \in (0,1)$. The formula (\ref{monotone density a}) still holds for $x>0$. The computation of the density function for $x<0$ is now delicate. 
For $x<0$, the formula (\ref{eq10}) still holds true. Note that now again $\theta -\alpha\pi \in(-\pi,0)$. By geometric consideration, 
the formula 
\begin{equation}
\begin{split}
\left( e^{i\alpha\pi}((-x)^\alpha + e^{i(\theta-\alpha\pi)})\right)_{(0,2\pi)}^{1/\alpha} 
&=  \left( e^{i\alpha\pi}\right)^{1/\alpha}\left((-x)^\alpha + e^{i(\theta-\alpha\pi)}\right)^{1/\alpha} \\
&=  -\left((-x)^\alpha + e^{i(\theta-\alpha\pi)})\right)^{1/\alpha},\qquad x<0, 
\end{split}
\end{equation}
is valid. A delicate point is that we should use the principal value in the last expression, not the branch $(\cdot)_{(0,2\pi)}^{1/\alpha}$.  Then the formula (\ref{formula1}) still holds and then the Stieltjes inversion formula implies (\ref{m4}). 

(iv)\,\, The case $\alpha\in(1,2)$, $\rho=1$. For $x>0$, the computation (\ref{monotone density a}) holds without any change. For $x<-1$, we have 
\begin{equation}\label{eq11}
\begin{split}
F(x+i0) &= ((x+i0)^\alpha +e^{i(\alpha-1)\pi} )^{1/\alpha}= ((-x)^\alpha e^{i\alpha\pi} -e^{i\alpha\pi})^{1/\alpha} \\
&= ((-x)^\alpha e^{i\alpha\pi} -e^{i\alpha\pi})^{1/\alpha} = -((-x)^\alpha -1)^{1/\alpha}<0,  
\end{split}
\end{equation}
and so $\mst_{\alpha,1}$ does not have support on $(-\infty,-1)$. We can also show that $\lim_{z\to-1, z\in\C^+}(z+1)G(z+1)=0$ which implies that  there is no atom at $-1$. For $x\in(-1,0)$, by using part of (\ref{eq11}), we have 
$$
F(x+i0) = \left(e^{i(\alpha-1)\pi}(1-(-x)^\alpha)\right)^{1/\alpha} = e^{i\frac{\alpha-1}{\alpha}\pi} (1-(-x)^\alpha)^{1/\alpha}, 
$$ 
and then by the Stieltjes inversion formula we get (\ref{m5}). 

The remaining statements can be proved as follows. 

(v)\,\, The reflection property of $\rho\mapsto 1-\rho$. For $\rho\in(0,1)$, we can prove the reflection property from (\ref{m2}) and (\ref{m4}). For $\rho=1$, we can take $\rho_n <1$ converging to $1$ and then the reflection property is justified since $\mst_{\alpha,\rho}$ is weakly continuous with respect to $(\alpha,\rho)\in((0,1)\cup (1,2])\times [0,1]$.  

(vi)\,\, The continuity of the density function at $x=0$. For $\alpha\in(0,1)\cup(1,2)$ and $\theta\in(0,\pi)$, we can show that 
$\lim_{x\searrow0}\varphi(x,\theta)=\theta$ and hence 
$$
\lim_{x\searrow0}M_\alpha(x,\theta)= \frac{\sin(\theta/\alpha)}{\pi}=\frac{\sin((\alpha\pi-\theta)/\alpha)}{\pi}=\lim_{x\searrow0}M_\alpha(x,\alpha\pi-\theta). 
$$ 
This shows the continuity.

The remaining statements are easy consequences of (\ref{m1})--(\ref{m6}). 
\end{proof}


\begin{thm} Let
 \begin{align}
&v_+:= \left(\frac{\sin(\theta-\frac{\alpha\pi}{1+\alpha})}{\sin(\frac{\alpha\pi}{1+\alpha})}\right)^{1/\alpha}, \qquad v_-:=-\left(\frac{\sin(\frac{\alpha^2\pi}{1+\alpha}-\theta)}{\sin(\frac{\alpha\pi}{1+\alpha})}\right)^{1/\alpha}. 
\end{align}
\begin{enumerate}[\rm(1)]
\item\label{mu1} If $\alpha\in(0,1)$, then $\mst_{\alpha,\rho}$ is unimodal. The mode is described as follows. 
 \begin{enumerate}[\rm (a)] 
\item\label{mu1a} If $\theta\in[0,\frac{\alpha^2\pi}{1+\alpha}]$, then the mode is $v_-$. 
 \item\label{mu1b} If $\theta\in [\frac{\alpha^2\pi}{1+\alpha}, \frac{\alpha\pi}{1+\alpha}]$, then the mode is $0$. 
\item\label{mu1c}  If $\theta\in[\frac{\alpha\pi}{1+\alpha},\alpha\pi]$, then the mode is $v_+$. 
\end{enumerate}

\item\label{mu2} If $\alpha=1$ and $\rho=\frac{1}{2}$, then $\mst_{1,1/2}$ is unimodal with mode 0. 

\item\label{mu3} If $\alpha \in(1,\frac{1+\sqrt{5}}{2}]$, then there are sub cases. 
\begin{enumerate}[\rm (a)] 
\item\label{mu3a} If $\theta\in[(\alpha-1)\pi,\frac{\alpha\pi}{1+\alpha}]$, then $\mst_{\alpha,\rho}$ is unimodal with mode $v_-$. 
\item\label{mu3c}  If $\theta\in(\frac{\alpha\pi}{1+\alpha},\frac{\alpha^2\pi}{1+\alpha})$, then $\mst_{\alpha,\rho}$ is bimodal with modes $v_-$ and $v_+$. 
 \item\label{mu3b} If $\theta\in [\frac{\alpha^2\pi}{1+\alpha}, \pi]$, then  $\mst_{\alpha,\rho}$ is unimodal with mode $v_+$. 

\end{enumerate}
\item\label{mu4} If $\alpha \in(\frac{1+\sqrt{5}}{2}, 2]$, then $\mst_{\alpha,\rho}$ is bimodal with modes $v_-$ and $v_+$. 
\end{enumerate}
Note that $\frac{1+\sqrt{5}}{2}=1.6180\ldots.$ 
\end{thm}
\begin{proof} We assume that $\alpha\neq 1,2$. 

(0) (Arguments valid for $\alpha\in(0,1)\cup(1,2),\rho\in(0,1)$)\,\, Let $p(x)$ be the density function of $\mst_{\alpha,\rho}$ and let $q(x,\theta):= x^2+2(\cos\theta)x +1$. The assumption $\rho\in(0,1)$ implies $\theta \in(0,\alpha\pi)$. Then $\mst_{\alpha,\rho}$ has the support $\R$ and $p(x)>0$ in $\R$. 
We can prove that  
\begin{equation}\label{derivative m}
p'(x)=\frac{\partial}{\partial x}M_\alpha(x,\theta) = -\frac{x^{\alpha-1}}{\pi q(x^\alpha,\theta)^\frac{\alpha+1}{2\alpha}}\sin\left(\frac{\alpha+1}{\alpha}\varphi(x^\alpha,\theta)\right),\qquad x>0.   
\end{equation}
 Since $x\mapsto\varphi(x,\theta)$ is strictly decreasing, mapping $(0,\infty)$ onto $(0,\theta)$, the following equivalence holds true: 
 
\begin{equation}\label{eq12}
\text{$p'(x)$ changes the sign at a point $x \in(0,\infty)$  $\Leftrightarrow\theta> \frac{\alpha\pi}{1+\alpha}$.} 
\end{equation} 
Moreover,  $\frac{\alpha+1}{\alpha}\theta\leq(\alpha+1)\pi<2\pi$ for $\alpha\in(0,1)$, and also $\frac{\alpha+1}{\alpha}\theta\leq(1/\alpha+1)\pi<2\pi$ for $\alpha\in(1,2)$, and so the sign of $p'(x)$ changes at most once in $(0,\infty)$. If the sign changes, the critical point is given by $x=v_+$. 

For the density function on the negative line, it suffices to study $\frac{\partial}{\partial x}M_\alpha(x,\alpha\pi-\theta)$, $x>0$, and it follows from (\ref{eq12}) that 
\begin{equation}\label{eq13}
\text{$p'(x)$ changes the sign at a point $x \in(-\infty,0)$  $\Leftrightarrow\theta< \frac{\alpha^2\pi}{1+\alpha}$.} 
\end{equation} 
The sign changes at most once, and if it changes, the critical point is $x=-v_-$. 

(i)\,\, The case $\alpha \in(0,1)$. The conditions (\ref{eq12}) and (\ref{eq13}) cannot be satisfied at the same time. Note that $p$ may have a mode at $0$. Combining these arguments and the geometric consideration of the graph of $p$, the claim (\ref{mu1}) follows for $\rho\in(0,1)$. The formula (\ref{derivative m}) holds also for $\theta=\alpha\pi$ and so the case $\rho=1$ is covered. The case $\rho=0$ is the reflection of $\rho=1$. 

(ii)\,\, The case $\alpha\in(1,2)$. If $\rho\in(0,1)$, then by looking at the formula (\ref{derivative m}) and the fact $\lim_{x\searrow0} \varphi(x,\theta)=\theta$, we have that $p'(+0)=0$, and from the replacement $\theta \mapsto \alpha\pi-\theta$ we have $p'(-0)=0$, and hence $p'(0)=0$. It is also true that $p'(0)=0$ for $\rho=0,1$. 

(ii-1)\,\, The case $\alpha \in(1,\frac{1+\sqrt{5}}{2}]$. Note that $(\frac{\alpha\pi}{1+\alpha},\frac{\alpha^2\pi}{1+\alpha}) \subset[(\alpha-1)\pi,\pi]$.   If $\theta \in (\frac{\alpha\pi}{1+\alpha},\frac{\alpha^2\pi}{1+\alpha})$, then $p$ has two modes at $x=v_-,v_+$ from (\ref{eq12}), (\ref{eq13}). The density function $p$ takes a local minimum at $x=0$, and hence we get (\ref{mu3c}). If $\theta\in((\alpha-1)\pi,\frac{\alpha\pi}{1+\alpha}]$, then $p'(x)=0$ only for $x=0,v_-$, and from geometric consideration of the graph of $p$, the sign of $p'$ does not change at $x=0$. For $\theta= (\alpha-1)\pi$ (or equivalently $\rho=1$), by using the formula (\ref{derivative m}) for $x>0$ and Proposition \ref{MD}(\ref{m5}) for $x\in(-1,0)$, we can show that $p$ is strictly decreasing in $(-1,\infty)$. Hence $\mst_{\alpha,1}$ is unimodal with mode $-1=v_-$. Thus we have (\ref{mu3a}). (\ref{mu3b}) is obtained by reflection. 

(ii-2)\,\, The case $\alpha \in(\frac{1+\sqrt{5}}{2}, 2)$. Note that $[(\alpha-1)\pi,\pi]\subset (\frac{\alpha\pi}{1+\alpha},\frac{\alpha^2\pi}{1+\alpha})$. So, if $\rho\in(0,1)$, $p'$ changes its sign at $x=v_+, v_-$, and also at $x=0$. If $\rho=1$, then $p'$ changes its sign at $x=v_+$ and 
$p$ is strictly decreasing in $(-1,0)$. From a geometric consideration, we find that $p$ takes a local minimum at 0. The case $\rho=0$ follows by refection. Hence we proved (\ref{mu4}). 
\end{proof}

\newpage

\begin{figure}[htbp]
\begin{center}
  \includegraphics[width=100mm]{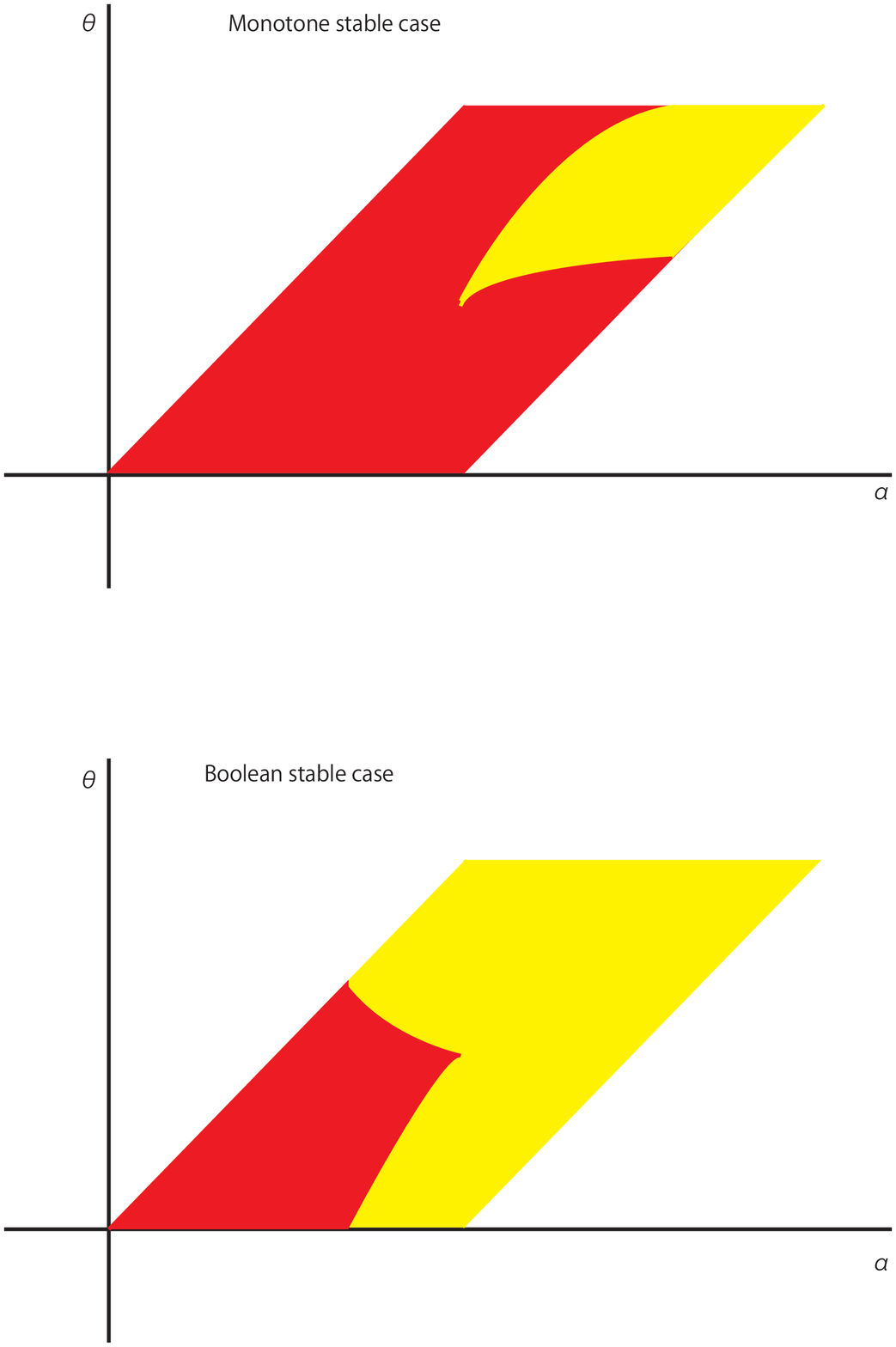}
\end{center}
\end{figure}

red: unimodal, yellow: bimodal


\begin{thebibliography}{99}

\bibitem[A65]{A65} N.I.\ Akhiezer, \textit{The Classical Moment Problem}, Oliver and Boyd, Edinburgh, 1965. 


\bibitem[AW13]{AW13} M.\ Anshelevich and J.\ D.\ Williams, Limit theorems for monotonic convolution and the Chernoff
product formula, to be published by Int.\ Math.\ Res.\ Not.\ arXiv:1209.4260

\bibitem[AH13]{AH13} O.\ Arizmendi and T.\ Hasebe, 
On a class of explicit Cauchy-Stieltjes transforms related to monotone stable and free Poisson laws, 
{\it Bernoulli} \textbf{19}(5B), 2013, 2750--2767. 


\bibitem[AH14]{AH} O.\ Arizmendi and T.\ Hasebe, Classical and free infinite divisibility for Boolean stable laws, {\it Proc.\ Amer.\ Math.\ Soc.} {\bf142} (2014), 1621--1632. 


\bibitem[BP99]{BP99} 
H.\ Bercovici and V.\ Pata, 
Stable laws and domains of attraction in free probability theory 
(with an appendix by P.\ Biane), 
{\it Ann.\ of Math.\ (2)}
\textbf{149}
(1999), 
1023--1060. 

\bibitem[B98]{B98} P.\ Biane,  
Processes with free increments,  
{\it Math.\ Z.}
\textbf{227} 
(1998), 
143--174.



\bibitem[H10]{H10} T.\ Hasebe, Monotone convolution and monotone infinite divisibility from complex analytic viewpoints, 
{\it Infin.\ Dimens.\ Anal.\ Quantum Probab.\ Relat.\ Top.} 
{\bf 13}, No.\ 1 (2010), 111--131. 

\bibitem[HT]{HT} T.\ Hasebe and S.\ Thorbj{\o}rnsen, Unimodality of the freely selfdecomposable probability laws, arXiv:1309:6776 

\bibitem[M01]{M01}
N.\ Muraki, Monotonic independence, monotonic central limit theorem and monotonic law of small numbers, {\it Infin.\ Dimens.\ Anal.\ Quantum Probab.\ Relat.\ Top.} \textbf{4} (2001), 39--58.




\bibitem[SW97]{SW97} R.\ Speicher and R.\ Woroudi, 
Boolean convolution, 
{\it Free Probability Theory, Ed.\ D.\ Voiculescu, Fields Inst.\ Commun.}
{\bf 12}, 
Amer. Math. Soc.,
(1997), 
267--280. 

 \bibitem[T00]{T00} G.\ Teschl,  \textit{Jacobi Operators and Completely Integrable Nonlinear Lattices}, Mathematical Surveys and Monographs \textbf{72}, Amer.\ Math.\ Soc., Providence, RI, 2000.

\bibitem[W12]{W12} J.-C.\ Wang, Strict limit types for monotone convolution,  
{\it J.\ Funct.\ Anal.}
{\bf 262}, no.\ 1 (2012), 35--58.

\bibitem[Y78]{Y78}
M.~Yamazato, {\it Unimodality of infinitely divisible
  distribution functions of class L}, Ann.~Probab.~{\bf 6} (1978),
523-531. 



\end{thebibliography}
\end{document}